\documentclass[12pt, a4paper]{article}
\usepackage{amsmath,amsfonts,amssymb, amsthm, xfrac}
\usepackage{enumerate}
\usepackage{mathrsfs}
\usepackage{graphicx}
\usepackage{tikz}

\newcommand{\N}{\mathbb{N}}

\newcommand{\F}{\mathbb{F}}
\newcommand{\Fq}{\mathbb{F}_q}
\newcommand{\Fqd}{\mathbb{F}_{q^2}}
\newcommand{\ord}{\mathop{\rm ord}}

\newcommand{\D}{\cdots}
\newcommand{\T}{\mathscr{T}}
\newcommand{\Zt}{\mathscr{Z}}
\newcommand{\G}{\mathscr{G}}
\newcommand{\V}{\mathcal{V}}
\newcommand{\E}{\mathcal{E}}

\newtheorem{Th}{Theorem}
\newtheorem{Co}[Th]{Corollary}

\newtheorem{Df}[Th]{Definition}
\newtheorem{Le}[Th]{Lemma}
\newtheorem{Ex}[Th]{Example}

\title{On the functional graph of $f(X)=c(X^{q+1}+aX^2)$ over quadratic extensions of finite fields}
\author{F. E. Brochero Mart\'inez and H. R. Teixeira }
\date{November 2021}

\begin{document}

\maketitle
\begin{abstract}
    Let $\Fq$ be the finite field with $q$ elements and $char(\Fq)$ odd. In this article we will describe completely the dynamics of the map $f(X)=c(X^{q+1}+aX^2)$, for $a=\{\pm1\}$ and $c\in\Fq^*$, over the finite field $\Fqd$, and give some partial results for $a\in\Fq^*\setminus\{\pm1\}$.
\end{abstract}

\section{Introduction}
Let $\Fq$ be a finite field with $q$ elements and $f:\Fq\to \Fq$ be a function. We define the functional graph of $f$ over $\Fq$ as the directed graph $\G(f)=(\V, \E)$, where $\V=\Fq$ and $\E=\{\langle x,f(x)\rangle\mid x\in\Fq\}$. We define the iterations of $f$ as $f^{(0)}(x)=x$ and $f^{(n+1)}(x)=f(f^{(n)}(x))$. Since $f$ is defined over a finite field, fixing $\alpha\in\Fq$, there are integers $0\leq i<j$, minimal, such that $f^{(i)}(\alpha)=f^{(j)}(\alpha)$. In the case when $i>0$, we call the list $\alpha,f(\alpha),f^{(2)}(\alpha),\D,f^{(i-1)}(\alpha)$ the pre-cycle and $f^{(i)}(\alpha),f^{(i+1)}(\alpha),\D,f^{(j-1)}(\alpha)$ the cycle of length $(j-i)$ or the $(j-i)$-cycle. If $\alpha$ is an element of a cycle, we call it a periodic element and, if $f(\alpha)=\alpha$, we say it is a fixed point.

The characteristics of functional graphs (number of cycles, cycle lengths, pre-cycle lengths and so on) have been studied for several different maps over finite fields, due to its applications in cryptography. Some quadratic maps have already been described, but mostly in the form $x^2+c$ as seen in \cite{PeMoMuYu}, \cite{Rog1} and \cite{VaSh}. L. Reis and D. Panario have described the functional graph associated to $\Fq$-linearized irreducible divisors of $x^n-1$ in \cite{PanRei}. Chow and Shparlinski obtained results on repeated exponentiation modulo a prime in $\cite{ChSh}$. Chebyshev polynomials are studied in \cite{Gas} and \cite{QurPan2} and Rédei functions in \cite{QurPan1}. Some rational maps have also been treated, such as $x+x^{-1}$ for $char(\Fq)=\{2,3,5\}$ in \cite{Ugo1} and rational maps induced by endomorphisms of ordinary elliptic curves in \cite{Ugo2}. Results on the iterations of each of the functions cited above and others have been gathered in a survey by Martins, Panario and Qureshi in $\cite{MaPaQu}$.

In this article we will discuss the characteristics of the functional graph of the map $X\mapsto c(X^{q+1}+aX^2)$ over the field $\Fqd$, where $c,a\in\Fq^*$. We will give the number of cycles of each length and the precise behavior of the pre-cycles for $a=\{\pm1\}$ and some partial results for the other cases.

\section{The functional graph of $f(X)=c(X^{q+1}-X^2)$}

First, we will analyse the case $a=-1$. Fixing an element $c$ in $\Fq^*$, lets establish the fixed points of $f(X)=c(X^{q+1}-X^2)$ over $\Fqd$. Observe that $cX(X^q-X)=X$ if, and only if, $X=0$ or $c(X^q-X)=1$. Therefore zero is a fixed point. If $\alpha\in\Fqd^*$ is such that $c(\alpha^q-\alpha)=1$, then $\alpha=\alpha^{q^2}=\alpha^q+c^{-1}=\alpha+2c^{-1}$. Since $char(\Fqd)\neq 2$, zero is the only fixed point.

For each element $\alpha\in\Fqd$, we define  $f^{-1}(\alpha)=\{\gamma\in\Fqd\mid f(\gamma)=\alpha\}$.

\begin{Df}
For every $\alpha\in\Fq$, let $\chi_2$ be the quadratic character over $\Fq$ defined as $$\chi_2(\alpha)=\begin{cases}1,& \text{if $\alpha$ is a square in $\Fq$}\\
-1,& \text{if $\alpha$ is not a square in $\Fq$}\\
0,& \text{if $\alpha=0$}.\end{cases}
$$
\end{Df}

\begin{Th}\label{prefq}
 If $\alpha\in\Fq$, then 
 
 \begin{enumerate}
     \item $\#f^{-1}(\alpha)=\begin{cases}
     2,& \text{if}\ \chi_2(-2\alpha c)=-1\\ 
     0,& \text{if}\  \chi_2(-2\alpha c)=1\\ 
     q,& \text{if $\alpha=0$}.  \end{cases}$
     
     \item  If $\gamma\in f^{-1}(\alpha)$, then $f^{-1}(\gamma)=\emptyset$.
 \end{enumerate}

\end{Th}
\begin{proof}
If $\alpha=0$, we have $c(X^{q+1}-X^2)=0$ if, and only if, $X\in\Fq$, then $\# f^{-1}(0)=q$.

Suppose $\alpha\in\Fqd^*$, and let $\{1,\beta\}$ be a base of $\Fqd$ over $\Fq$, such that $\beta^2=b\in\Fq^*$. Then, if $X\in f^{-1}(\alpha)$, writing $X=x+y\beta$, with $x,y\in\Fq$, we have
\begin{align*}
    \alpha&=c(X^{q+1}-X^2)\\
    &=c(x+y\beta)(x+y\beta^q-x-y\beta)\\
    &=c(x+y\beta)y(\beta^q-\beta)\\
    &=cxy(\beta^q-\beta)-cy^2(\beta^{q+1}-\beta^2).
\end{align*}
Since $\beta^2=b$ is not a square in $\Fq$, we have $\beta^q=b^{\frac{q-1}{2}}\beta=-\beta$ and $\beta^{q+1}=-b$. Then
\begin{equation}
    -2cxy\beta-2bcy^2=\alpha.\label{f1}
\end{equation}

If $\alpha\in\Fq^*$, then $-2cxy=0$ and $-2bcy^2=\alpha$, which implies that $x=0$ and $y^2=-\frac{\alpha}{2bc}$. Since $b$ is not a square, if $\chi_2(-2\alpha c)=1$, then $f^{-1}(\alpha)=\emptyset$. On the other hand, if $\chi_2(-2\alpha c)=-1$, we have $f^{-1}(\alpha)=\{\pm k\beta\}$, where $k^2=-\frac{\alpha}{2bc}$.

Moreover, taking $\gamma\in f^{-1}(\alpha)$, we have $\gamma=k\beta$ for some $k\in\Fq$. By $(\ref{f1})$, if $x\in f^{-1}(\gamma)$, then $$-2cxy\beta-2bcy^2=k\beta,$$ consequently $-2bcy^{2}=0$ and $-2cxy=k$. Since $-2bc\neq 0$, there are no $x$ and $y$ that satisfy the equations and $f^{-1}(\alpha)=\emptyset$.
\end{proof}

With this theorem we can characterize the first connected component of the functional graph of $X\mapsto c(X^{q+1}-X^2)$. We have zero directed to itself, then the $q-1$ elements of $\Fq^*$ directed to zero and two elements of $\Fqd\setminus\Fq$ attached to each of the $\frac{q-1}{2}$ elements of $\Fq^*$ that satisfy $\chi_2(-2\alpha c)=-1$.

\begin{Ex}
Taking $q=13$ and $c=3$, let $\{1,\beta\}$ be a base for $\mathbb{F}_{13^2}$ over $\mathbb{F}_{13}$, such that $\beta^2=11$. Then the connected component that contains the vertex $\langle 0,0\rangle\in\mathbb{F}_{13}\oplus \beta \mathbb{F}_{13}$ is the following 
\begin{center}
\includegraphics[scale=0.6]{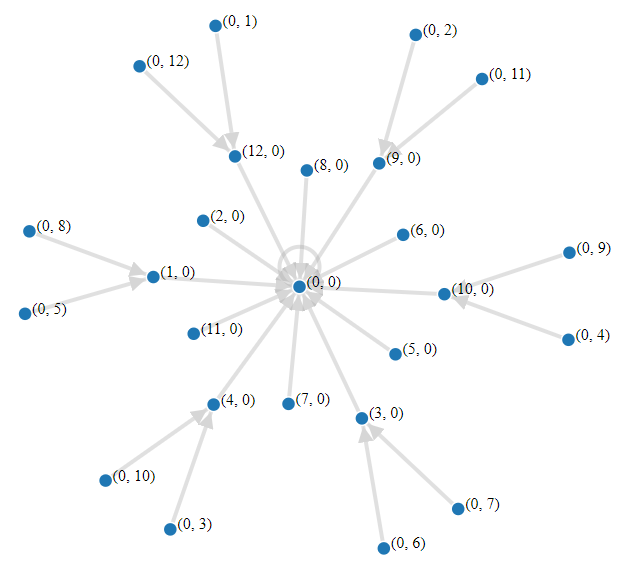}
\end{center}
\end{Ex}

\begin{Df}
For $i,m$ and $n$, positive integers, let 
\begin{enumerate}[a)]
    \item $Cyc(n)$ be the cyclic directed graph of length n;
    \item $\T(1)$ be the tree composed by two points, $P_1$ and $P$, where $P_1$ is directed to $P$. For $m\ge 1$, $\T(m+1)$ is the tree obtained after attaching $2$ points directed to each point in the last level of $\T(m)$;
    \item $(Cyc(n),\T(m))$ be the graph obtained after replacing each point of $Cyc(n)$ by a tree isomorphic to $\T(m)$.
    \item $\Zt(q)$ be the graph $Cyc(1)$ attached to $\frac{q-1}{2}$ trees isomorphic to $\T(1)$ and $\frac{q-1}{2}$ trees isomorphic to $\T(2)$.
\end{enumerate}
\end{Df}

Note that the connected component  of the functional graph of $f(X)=c(X^{q+1}-X^2)$ that contains zero is precisely the graph $\Zt(q)$.

\begin{Le}\label{zd}
 If $\alpha\in\Fqd\setminus\Fq$ then $\# f^{-1}(\alpha)=\{0,2\}$.
\end{Le}
\begin{proof}
    We can write $\alpha$ as $\alpha=u+v\beta$, where $u,v\in\Fq$ and $v\neq 0$. We want to know if there is a solution to $c(X^{q+1}-X^2)=\alpha$ in $\Fqd$. Writing $X=x+y\beta$, by $(\ref{f1})$, we just need to solve $$-2bcy^2-2cxy\beta=u+v\beta.$$ Observe that, if there is $y\in\Fq$ such that $y^2=-\frac{u}{2bc}$, than we take $x=-\frac{v}{2cy}$ and we have a solution. 

    If $u=0$, we have proved that $\# f^{-1}(\alpha)=0$. If $u$ is such that $\chi_2(-2bcu)=-1$, then there is no solution and $\# f^{-1}(\alpha)=0$. Conversely, if $\chi_2(-2bcu)=1$, then there exists $y\in\Fq$ such that $y^2=-\frac{u}{2bc}$. Note that $-y$ is also a solution, and since $y^2=-\frac{u}{2bc}$ has at most two solutions, those are the only ones. In conclusion, $\# f^{-1}(\alpha)=\{0,2\}$.
\end{proof}

Suppose that $\alpha\in \Fqd\setminus(\Fq\cup\beta\Fq)$ and that $x+y\beta\in f^{-1}(\alpha)$, hence $-x-y\beta\in f^{-1}(\alpha)$. 
If $q\equiv3\pmod4$, notice that $\chi_2(-1)=-1$, ergo if $\chi_2(-2bcx)=-1$, then $\chi_2(-2bc(-x))=1$. In other words, in 
the preimage of $\alpha$, there is one element that has two preimages and one that has none. 

Let $\alpha\in\Fqd$ be non periodic and $j\in\N$ be the minimum such that $a_j:=f^{(j)}(\alpha)$ is a periodic point. As $a_j$ is an element of a cycle, there is a periodic point $w$ such that $f(w)=a_j$, hence $w$ has preimage. Consequently $a_{j-1}$ does not have preimage and $a_{j-1}=\alpha$. In conclusion, if $q\equiv3\pmod4$, for each connected components different from $\Zt(q)$ there exists $n\in\N$ such that the component is  isomorphic to $(Cyc(n), \T(1))$.

\begin{Ex}
Taking $q=7$ and $c=4$, the two distinct cycles other than $\Zt(q)$ that appear in the functional graph are isomorphic to the following
\begin{center}
\includegraphics[scale=0.4]{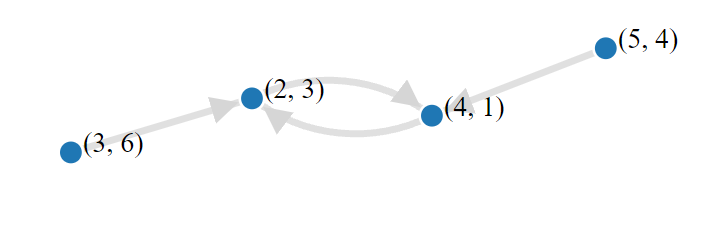}
\includegraphics[scale=0.6]{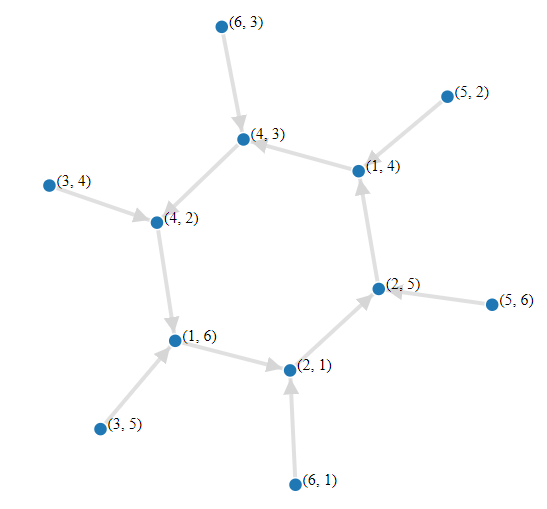}
\end{center}
which are $(Cyc(2), \T(1))$ and $(Cyc(6), \T(1))$ respectively. 
\end{Ex}

If $q\equiv1\pmod 4$, then $\chi_2(-1)=1$. Hence either $\chi_2(-2bcx)=-1$ and $\chi_2(-2bc(-x))=-1$, or  $\chi_2(-2bcx)=1$ and  $\chi_2(-2bc(-x))=1$. Which means that  all the elements of the preimage of $\alpha$ have a preimage or none of them do.

\begin{Ex}
Taking $q=13$ and $c=3$, the connected components different from $\Zt(q)$ are isomorphic to

\begin{center}
\includegraphics[scale=0.5]{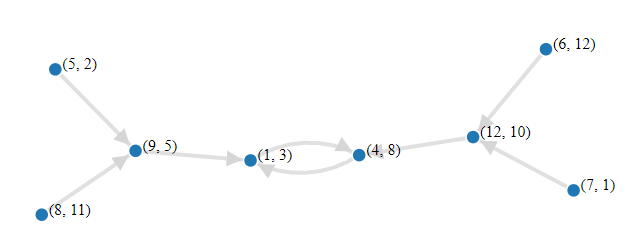}
\includegraphics[scale=0.7]{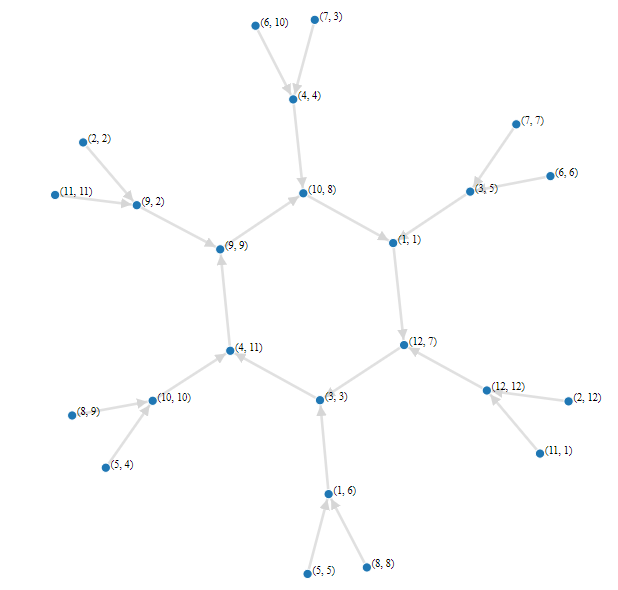}
\end{center}
which are $(Cyc(2), \T(2))$ and $(Cyc(6), \T(2))$ respectively. 
\end{Ex}

We will now find the number of cycles of each length, but first we need the following lemma.

\begin{Le}\label{EvenCycle}
 Every cycle length greater than $1$ is even.
\end{Le}
\begin{proof}
    We have proved that if $X=x+y\beta\in\Fq$ and $x=0$ or $y=0$, then $X$ is attached to a cycle of length $1$. Therefore, suppose $x,y\neq 0$. By $(\ref{f1})$, we have $f(x+y\beta)=-2bcy^2-2cxy\beta$. Then, looking at $\Fqd$ as a vector space over $\Fq$ with base $\{1,\beta\}$, we have 
    \begin{align}
    (\Fq^*)^2&\overset{f}{\to} (\Fq^*)^2\nonumber\\
    \langle x,y\rangle &\mapsto \langle -2bcy^2,-2cxy\rangle \label{f1v}
    \end{align}
    and, applying $f$ again,
    \begin{align}
    f^{(2)}(x,y)&=\langle -8bc^3x^2y^2,-8bc^3xy^3\rangle \nonumber\\
    &=-8bc^3xy^2\langle x,y\rangle .\label{f2v}
\end{align}
    Writing $g(x,y)=-8bc^3xy^2$, for all $v\in\Fq$, we get that \begin{align*}
        f^{(2)}(vx,vy)&=g(vx,vy)\langle vx,vy\rangle \nonumber\\
        &=v^4g(x,y)\langle x,y\rangle .
    \end{align*}
Furthermore, taking $v=g(x,y)$, we have\begin{align*}
        f^{(4)}(x,y)&=f^{(2)}(g(x,y)\langle x,y\rangle )\\&=g(x,y)^5\langle x,y\rangle . \end{align*}
Therefore, by induction, we conclude that \begin{align}
    f^{(2n)}(x,y)&=g(x,y)^{1+4+4^2+\D+4^{n-1}}\langle x,y\rangle \nonumber\\
    &=g(x,y)^{\frac{4^{n}-1}{3}}\langle x,y\rangle \label{f2nv}.
\end{align}
Applying $f$ once more, it follows that \begin{align*}
    f^{(2n+1)}(x,y)&=f(g(x,y)^{\frac{4^{n}-1}{3}}\langle x,y\rangle )\\
    &=k_{x,y,n}\langle by,x\rangle ,
\end{align*}
where $k_{x,y,n}=-2cy_1g(x,y)^{\frac{2(4^{n}-1)}{3}}$. Notice that $k_{x,y,n}\in\Fq$. 

If there is a cycle of odd length, then we must have $k_{x,y,n}\langle by,x\rangle =\langle x,y\rangle $, for some $\langle x,y\rangle \in(\Fq^*)^2$, which implies that  $\frac{by}{x}=\frac{x}{y}$ and $b=(\frac{x}{y})^2$. Since $b$ is not a square in $\Fq$, we are through.
\end{proof}

\begin{Th}\label{T9}
 Let $q-1=2^sr$, with $r$ odd. Then for every $d$, divisor of $r$, there are $\frac{\varphi(d)(q-1)}{2\ord_{3d}(4)}$ cycles of length $2\ord_{3d}(4)$, and those are the only cycles of length greater than $1$.
\end{Th}
\begin{proof}
    Suppose $\langle x,y\rangle $ is an element of a cycle of length $m>1$. Then, by Theorem $\ref{EvenCycle}$, $m=2n$ for some $n\in\N$ and $f^{(2n)}(X)=X$, which implies that $(-8bc^3xy^2)^\frac{4^{n}-1}{3}=1$. Notice that, if $-8bc^3xy^2=\theta\in\Fq^*$, then $n$ is the minimum such that $\frac{4^{n}-1}{3}\equiv0\pmod{\ord(\theta)}$, which is equivalent to $4^{n}\equiv1\pmod{3\ord(\theta)}$. Since there exists such $n$ if, and only if, $\ord(\theta)$ is odd, then, for every $2n$-cycle, there is a $d$, odd divisor of $q-1$, i.e. divisor of $r$, such that $n=\ord_{3d}(4)$. Now we will count the number of cycles of such length.
    
    Now, suppose that $\alpha$ is a generator of $\Fq^*$ and $\alpha^j$ is an element of order $d$. Since
    $$\ord(\alpha^j)=\frac{q-1}{\gcd(q-1,j)}:=d,$$ 
    then $2^s$ divides $\gcd(q-1,j)$, hence $j=2^st$ with $1\le t\le \frac{q-1}{2^s}$. 
    In addition, for any $\langle x,y\rangle\in \Fq^2$, such that $-8bc^3xy^2=\alpha^j$, the element $x+\beta y\in \Fqd$ is a node of a $2\ord_{3d}(4)$-cycle. 
    
    Notice that $t$ satisfies $\gcd(\frac{q-1}{2^s},t)=\frac{q-1}{2^sd}$ if, and only if, $t=\frac{q-1}{2^sd}t'$ with $\gcd(t',d)=1$. Since $1\leq t\leq \frac{q-1}{2^s}$, then $1\leq t'\leq d$. Hence there are $\varphi(d)$ choices for $t$, and thus $\varphi(d)$ choices for $j$ such that $\ord(\alpha^j)=d$. 
    
    Furthermore, for every $j$, there are $q-1$ ordered pairs $\langle x,y\rangle $ such that $-8bc^3xy^2=\alpha^j$. Consequently the number of elements in $2\ord_{3d}(4)$-cycles is $\varphi(d)(q-1)$. Since there are $2\ord_{3d}(4)$ elements in each cycle, there are $\frac{\varphi(d)(q-1)}{2\ord_{3d}(4)}$ cycles.
\end{proof}
\begin{Co}
 There are $\frac{q-1}{2}$ cycles of length two.
\end{Co}
\begin{proof}
    Since $1$ is always an odd divisor of $q-1$, taking $d=1$ we have $2\ord_{3d}(4)=2\ord_3(4)=2$. Notice that if $d\geq 3$ is a divisor of $r$, then $3d\geq 9$ and $2\ord_{3d}(4)\neq 2$, as $\ord_{3d}(4)>1$.
    Therefore, the number of $2$-cycles is $\frac{\varphi(1)q-1}{2}=\frac{q-1}{2}$.
\end{proof}

Now that we have found every cycle length of this functional graph we only need to describe the pre-cycles.
\begin{Df}\label{cycg}
Let $Cyc(\G)= (\V,\E)$ be the subgraph of $\G\setminus\Zt(q)$ such that $\V$ is the set of periodic points and $\E=\{\langle x,f(x)\rangle | x\in \V\}$.
\end{Df}
\begin{Th}\label{precyc}
 If $q-1=2^sr$, then, for every node of $Cyc(\G)$, there is a tree isomorphic to $\T(s)$ attached to it. 
\end{Th}
\begin{proof}
    Let  $d$ be a divisor of $r$. Once again let $\Fqd$ be generated by $\{1,\beta\}$ over $\Fq$, where $\beta^2=b\in\Fq$. We will look at $\Fqd$ as a vector space over $\Fq$. Let $\langle x_0,y_0\rangle$ be a periodic element and, for all $i\geq 1$, let $\langle x_i,y_i\rangle$ be a periodic point such that $$f(x_i,y_i)=\langle x_{i-1},y_{i-1}\rangle.$$ As seen in Lemma $\ref{zd}$, $\langle -x_i,-y_i\rangle $ also satisfies $f(-x_i,-y_i)=\langle x_{i-1},y_{i-1}\rangle$, but is not a periodic element. Suppose that $\langle x,y\rangle \in f^{-1}(-x_1,-y_1)$.
     
    \begin{center}
        \includegraphics[scale=0.6]{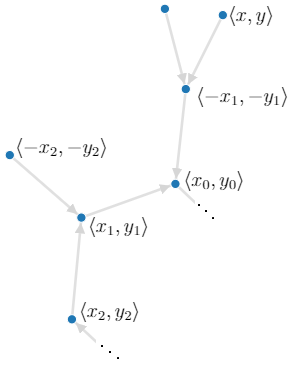}
    \end{center}
    
    Using $(\ref{f1v}),$ we have $-2bcy^2=-x_1$ and $-2cxy=-y_1$. But $-2bcy_2^2=x_1$ and $-2cx_2y_2=y_1$. Then, $-2bcy^2=2bcy_2^2$ and $-2cxy=2cx_2y_2$, which implies that $y^2=-y_2^2$. Therefore, $\langle -x_1,-y_1\rangle $ has nonempty preimage if, and only if, $-1$ is a square in $\Fq$ or, equivalently, if, and only if, $q\equiv 1\pmod 4$. If $q\equiv 1\pmod 4$, then $y=\pm\zeta_4y_2$ and $x=\pm\zeta_4x_2$, where $\zeta_j$ is a primitive $j$-th root of unity in $\Fq$.
    
    Suppose, by induction hypothesis, that for each $0<k\le s$, 
    the pair $\langle \zeta_{2^{k-1}}x_{k-1},\zeta_{2^{k-1}}y_{k-1}\rangle$ is a non periodic element that satisfies $$f^{(k-1)}(\zeta_{2^{k-1}}x_{k-1},\zeta_{2^{k-1}}y_{k-1})=\langle x_0,y_0\rangle.$$ If $\langle x,y\rangle \in f^{-1}(\zeta_{2^{k-1}}x_{k-1},\zeta_{2^{k-1}}y_{k-1})$, then $-2bcy^2=\zeta_{2^{k-1}}x_{k-1}$ and, since $2bcy_k^2=x_{k-1}$ by definition, we have $2bcy^2=\zeta_{2^{k-1}}2bcy_k^2$. It follows that $y^2=\zeta_{2^{k-1}}y_k^2$ and, consequently, $f^{-1}( \zeta_{2^{k-1}}x_{k-1},\zeta_{2^{k-1}}y_{k-1})\neq\emptyset$ if, and only if, $\zeta_{2^{k-1}}$ is an square in $\Fq$, that is equivalent to $q\equiv 1\pmod {2^k}$.
    
    This shows that, if $q-1=2^sr$, where $r$ is odd, then, for every $\langle x_0,y_0\rangle $ in $Cyc(\G)$, there is a tree $\T(k)$ attached to it. Since this goes for every element of every cycle, this concludes our proof.
\end{proof}

Now, we have the whole functional graph of $f(X)=c(X^{q+1}-X^2)$ over $\Fqd$. First, there is a component isomorphic to $\Zt(q)$, where zero is the fixed point, and writing $q-1=2^sr$, for every $d$ divisor of $r$, there are $\frac{\varphi(d)(q-1)}{2\ord_{3d}(4)}$ connected components  isomorphic to $(Cyc(2\ord_{3d}(4)),\T(s))$. In conclusion we have the following theorem

\begin{Th}
Let $q$ be a power of an odd prime with $q-1=2^sr$ and $r$ odd.  The functional graph of the function $f(X)=c(X^{q+1}-X^2)$ over $\Fqd$, with $c\in \Fq$, is isomorphic to 
    $$\G=\Zt(q)\bigoplus_{d\mid r}\frac{\varphi(d)(q-1)}{2\ord_{3d}(4)}\times\big(Cyc(2\  \ord\nolimits_{3d}(4)),\T(s)\big).$$
\end{Th}

\begin{Ex}
Once more, let $q=13$ and $c=3$. Notice that $q-1=2^2\cdot 3$, then we have $\T(2)$ attached to each periodic point. Since the odd divisors of $q-1$ are $1$ and $3$, the cycle lengths must be $2\ord_{3}(4)=2$ and $2\ord_{9}(4)=6$. Finally, we conclude that the functional graph is composed by one component isomorphic to
\begin{center}
\includegraphics[scale=0.4]{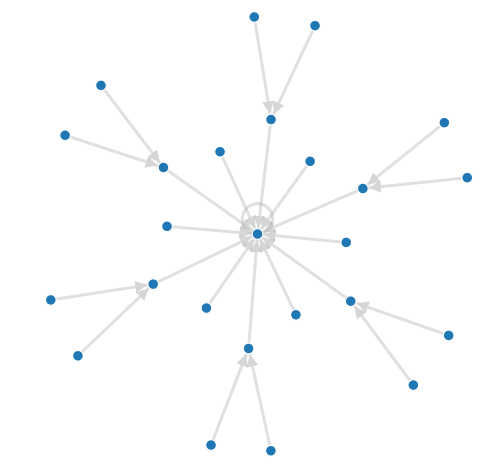}    
\end{center}
 $\frac{q-1}{2}=6$ components isomorphic to
\begin{center}
\includegraphics[scale=0.4]{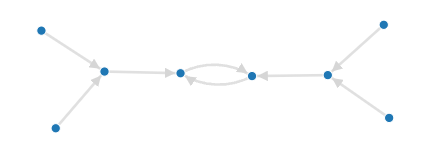}    
\end{center}
and $\frac{\varphi(d)(q-1)}{6}=\frac{2\cdot12}{6}=4$ components isomorphic to 
\begin{center}
\includegraphics[scale=0.4]{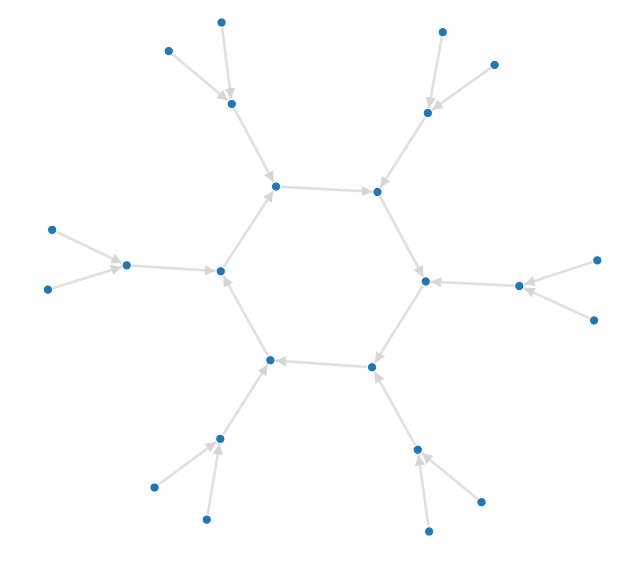}
\end{center}
In other words, the functional graph of $f(X)=3(X^{14}-X^2)$ over $\mathbb{F}_{13^2}$ is $$\G=\Zt(13)\oplus 4\times(Cyc(6),\T(2))\oplus 6\times(Cyc(2),\T(2)).$$

\end{Ex}

\section{The functional graph of $f(X)=c(X^{q+1}+X^2)$}

Our second case will be for $a=1$. We will follow the same path as the first case, establishing the connected component that contains zero, then the cycle lengths and number of cycles and, finally, the behavior of the pre-cycles.

Observe that, writing $X=x+y\beta$, we have \begin{align*}
    f(X)&=c(x+y\beta)(x+y\beta^q+x+y\beta)\\
    &=c(x+y\beta)(x-y\beta+x+y\beta)\\
    &=2cx^2+2cxy\beta,
\end{align*}
since $\beta^q=-\beta$. Then, looking at $\Fqd$ as a vector space over $\Fq$ we get \begin{align}
    (\Fq)^2&\to(\Fq)^2\nonumber\\
    \langle x,y\rangle&\mapsto\langle 2cx^2,2cxy\rangle.\label{f5}
\end{align} 

A direct consequence is that $\langle x,y\rangle $ is a fixed point if, and only if, $x=y=0$ or $x=(2c)^{-1}$ and $y$ is an arbitrary element of $\Fq$. Thus, there are $q+1$ fixed points: zero and $\langle (2c)^{-1},y\rangle ,$ for all $y\in\Fq$.

Observe that $f(X)=0$ if, and only if, $x=0$. In other words, $f^{-1}(0)=\{\langle 0,y\rangle \mid y\in\Fq\}$, hence $\#f^{-1}(0)=q$. It follows directly from $(\ref{f5})$ that $\#f^{-1}(0,y)=\emptyset$, if $y\neq0$. In conclusion, we have a theorem analogue to Theorem \ref{prefq}.

\begin{Th}
 If $\alpha\in\beta\Fq$, then $$\#f^{-1}(\alpha)=\begin{cases}q,&\text{if}\ \alpha=0\\ 0,&\text{if}\ \alpha\neq 0.\end{cases}$$
\end{Th}

Therefore, the first connected component of the functional graph of the map $X\mapsto c(X^{q+1}+X^2)$ is zero directed to itself and the $q-1$ elements of $\Fq^*$ directed to zero, which is isomorphic to the graph of the following definition.

\begin{Df}
We define $\Zt^*(q)$ as the directed graph $Cyc(1)$ with $q-1$ trees isomorphic to $\T(1)$ attached to it. 
\end{Df}

Now we will count the existing cycles.

\begin{Th}
 Let $q-1=2^sr$, with $r$ odd. Then for every $d\neq 1$, divisor of $r$, there are $\frac{q\cdot\varphi(d)}{\ord_d(2)}$ cycles of length $\ord_d(2)$, and those are all the cycles of length greater than 1. For $d=1$ there are $\frac{q\cdot\varphi(d)}{\ord_d(2)}+1$ cycles of length $\ord_d(2)$.
\end{Th}
\begin{proof}
Using $(\ref{f5})$, we have that $f^{(2)}(x,y)=\langle 8x^4,8x^3y \rangle$, then, for $n\in\N$, \begin{align}
    f^{(n)}(x,y)=&\langle (2c)^{2^n-1}x^{2^n},(2c)^{2^n-1}x^{2^n-1}y\rangle\nonumber\\
    =&(2cx)^{2^n-1}\langle x,y \rangle.\label{f6}
\end{align}
If $\langle x,y\rangle$ is an element of a cycle of length $n\geq1$, different from $\langle 0,0\rangle$, then $f^{(n)}(x,y)=\langle x,y\rangle$ and $n$ is the minimal such that $(2cx)^{2^n-1}=1$. As in Theorem \ref{T9}, writing $2cx=\theta$, we have $2^n\equiv1\pmod{\ord(\theta)}$, but there exists such $n$ if, and only if, $\ord(\theta)$ is odd. Therefore, for every $n$-cycle, there is a $d$, odd divisor of $q-1$ such that $n=\ord_d(2)$. 

If $\alpha$ is a generator of $\Fq^*$, then, as seen in Theorem $\ref{T9}$, there are $\varphi(d)$ choices for $j$ such that $\ord(\alpha^j)=d$ and $q$ ordered pair $\langle x,y\rangle$ such that $2x=\alpha^j$. Consequently the number of elements in $\ord_d(2)$-cycles is $q\cdot\varphi(d)$, hence, there are $\frac{q\cdot\varphi(d)}{\ord_d(2)}$ cycles of length $\ord_d(2)$. Note that, if $d=1$, then  $\ord_1(2)=1$ and we have $\frac{q\cdot\varphi(d)}{\ord_d(2)}$ cycles of length one plus the cycle of the element zero.
\end{proof}

Finally, we only need to understand the behavior of the pre-cycles. With following lemma we will conclude that the pre-cycles of this functional graph are exact the same as the previous case.

\begin{Le}
 If $\alpha\in\Fqd\setminus\beta\Fq$, then $\#f^{-1}(\alpha)=\{0,2\}$.
\end{Le}
\begin{proof}
If $\alpha=u+v\beta$ and $X=x+y\beta\in f^{-1}(\alpha)$, then $$2cx^2+2cxy\beta=u+v\beta.$$
If $u=0$ or $\chi_2(2cu)=-1$, then $f^{-1}(\alpha)=\emptyset$. If $\chi_2(2cu)=1$, then there is $x\in\Fq$ such that $x^2=\frac{u}{2c}$ and $y=\frac{v}{2cx}$. Since $\pm x$ are both solutions, those are the only ones, thus $\#f^{-1}(\alpha)=2$.
\end{proof}

Using $Cyc(\G)$ as Definition $\ref{cycg}$, where $\G$ is now the functional graph of the map $f(X)=c(X^{q+1}+X^2)$, the following theorem follows with the same proof as Theorem \ref{precyc}, just switching $-2bcy^2$ for $2cx^2$.

In conclusion, writing $q-1=2^sr$, with $r$ odd, the functional graph of the map $c(X^{q+1}+X^2)$ has one component isomorphic to $\Zt^*(q)$, where zero is the fixed point, and, for each $d$ divisor of $r$, $\frac{q\cdot\varphi(d)}{\ord_d(2)}$ components isomorphic to $(Cyc(\ord_d(2)),\T(s))$. In other words, 

\begin{Th}
    If $q-1=2^sr$, with $r$ odd, then 
    the functional graph of the function $f(x)=c(x^{q+1}+x^2)$ over $\Fqd$ is isomorphic to 
$$\G=\Zt^*(q)\bigoplus_{d\mid r}\frac{q\cdot \varphi(d)}{\ord_{d}(2)}\times\big(Cyc(\ord\nolimits_{d}(2)),\T(s)\big).$$
\end{Th}

\begin{Ex}
For $q=13$ and $c=3$ we have the following: one component isomorphic to 
\begin{center}
\includegraphics[scale=0.3]{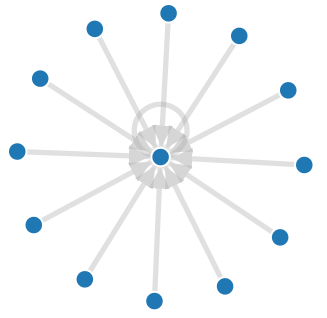}
\end{center}
for $d=1$, we have $13$ components isomorphic to 
\begin{center}
    \includegraphics[scale=0.3]{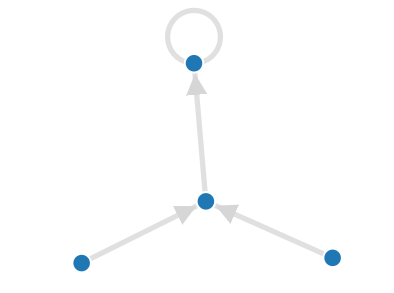}
\end{center}
and, for $d=3$, we have $\frac{13\cdot\varphi(3)}{\ord_3(2}=13$ components isomorphic to 
\begin{center}
\includegraphics[scale=0.5]{Graph213nolabel.png}
\end{center}
Equivalently, the functional graph of $f(X)=3(X^{14}+X^2)$ over $\mathbb{F}_{13^2}$ is $$\G=\Zt^*(13)\oplus13\times(Cyc(1),\T(2))\oplus13\times(Cyc(13),\T(2)).$$
\end{Ex}

\section{Partial results for $a\neq\{\pm1\}$}

In this case we will assume $c=1$ for notation purposes. The general case is equally obtained. 

When $a\neq \{\pm1\}$, writing $X=x+y\beta$, we have 
\begin{align*}
    f(X)&=(x+y\beta)(x-y\beta +a(x+y\beta))\\
    &=(x+y\beta)((1+a)x+(a-1)y\beta)\\
    &=((a+1)x^2+(a-1)by^2)+2axy\beta.
\end{align*}
Therefore, writing $\Fqd$ as a vector space over $\Fq$, we have
\begin{align}
    (\Fq)^2\to& (\Fq)^2\nonumber\\
    \langle x,y\rangle\mapsto&\langle(a+1)x^2+(a-1)by^2, 2axy\rangle\label{ff}.
\end{align}
Since $a\neq\pm1$, we conclude that $f^{-1}(0,0)=\langle 0,0\rangle$, ergo, zero is a fixed point without a pre-cycle.

Notice that if $\langle x,y\rangle$ is another fixed point, then it must satisfy the following equations \begin{center}
    $(a+1)x^2+(a-1)by^2=x$ and $2axy=y$.
\end{center}
If $y\neq 0$, the second equation implies $x=\frac{1}{2a}$. Consequently, by the first one $\frac{a+1}{4a^2}+(a-1)by^2=\frac{1}{2a}$, and $y^2=\frac{1}{4ba^2}$. As $b$ is not a square in $\Fq$, this equation has no solution. If $y=0$, then $x=\frac{1}{a+1}$. Hence $\langle\frac{1}{a+1},0\rangle$ and $\langle0,0\rangle$ are the only fixed points. 

Now we will look at the preimage of two distinct types of points: $\alpha\in\Fq$ and $\alpha\in\beta\Fq$. 
\begin{Th}\label{pre}    
For $\alpha\in \Fq^*$,
    \begin{enumerate}
        \item $\#f^{-1}(\alpha)=\begin{cases}0,&\text{if}\ \chi_2(\alpha(a-1))=1\ \text{and}\ \chi_2(\alpha(a+1))=-1\\
        4,&\text{if}\ \chi_2(\alpha(a-1))=-1\ \text{and}\ \chi_2(\alpha(a+1))=1\\
        2,&\text{otherwise}
        \end{cases}$
        
        \item if $q\equiv3\pmod4$, then $$\#f^{-1}(\alpha\beta)=\begin{cases}
        0,&\text{if}\ \chi_2(-a^2+1)=1\\
        2,&\text{if}\ \chi_2(-a^2+1)=-1\end{cases}$$
        
        \item if $q\equiv1\pmod4$, then $$\#f^{-1}(\alpha\beta)=\begin{cases}
        4,&\text{if}\ \chi_2(-a^2+1)=-1\ \text{and}\ \chi_2(2\gamma_a\alpha a)=-1\\
        0,&\text{otherwise},
        \end{cases}$$
    \end{enumerate}
    where $\gamma_a^2=-\frac{(a-1)b}{a+1}$.
\end{Th}
\begin{proof}
If $\alpha\in\Fq$, then  $\langle x,y \rangle\in f^{-1}(\alpha)$ implies that \begin{center}
    $(a+1)x^2+(a-1)by^2=\alpha$ and $2axy=0$.
\end{center}

If $x=0$, we have $y^2=\frac{\alpha}{(a-1)b}$ which has two solutions if $\chi_2(\alpha(a-1))=-1$ and zero otherwise. If $y=0$, we have $x^2=\frac{\alpha}{a+1}$ which has two solutions if $\chi_2(\alpha(a+1))=1$ and zero otherwise.

Conversely, $\alpha\beta=\langle 0,\alpha\rangle$, then if $\langle x,y \rangle\in f^{-1}(\alpha\beta)$, we have  $(a+1)x^2+(a-1)by^2=0$ and $2axy=\alpha$ or, equivalently, \begin{center}
    $\big(\frac{x}{y}\big)^2=-\frac{(a-1)b}{a+1}$ and $xy=\frac{\alpha}{2a}$.
\end{center}  The first equation has a solution if, and only if, $\chi_2(-a^2+1)=-1$. If $\gamma_a$ is such that $\gamma_a^2=-\frac{(a-1)b}{a+1}$, then $\frac{x}{y}=\pm\gamma_a$ and, multiplying by the second equation, $x^2=\pm\frac{\gamma_a\alpha}{2a}$. 

Note that if $q\equiv3\pmod4$, then $x^2=\pm\frac{\gamma_a\alpha}{2a}$ has two solutions, since $-1$ is not a square and we can choose the signal in order for it to have solutions. On the other hand, if $q\equiv1\pmod4$, then $-1$ is a square and either $\chi_2(2\gamma_a\alpha a)=1$ and $x^2=\pm\frac{\gamma_a\alpha}{2a}$ has four solutions or $\chi_2(2\gamma_a\alpha a)=-1$ and $x^2=\pm\frac{\gamma_a\alpha}{2a}$ has no solutions, and this concludes our proof.
\end{proof}

Restricting $f$ to $\Fq$, notice that \begin{align*}
    f|_{\Fq}:\Fq&\to\Fq\\
    x&\mapsto (a+1)x^{2}.
\end{align*}
The functional graph of $x\mapsto x^2$ over $\Fq$ has been studied and described in several different papers, as $\cite{ChSh}$ and $\cite{VaSh}$. Since its behavior is known, here it will be denoted simply as $\G(f|_{\Fq})$.

When $\chi_2(1-a^2)=1$, with the previous theorem we conclude that, for every point $\alpha\in\G(f|_{\Fq})$, if $\chi_2(\alpha(a-1))=1$, then there are no points outside of $\G(f|_{\Fq})$ in its preimage. On the other hand, if $\chi_2(\alpha(a-1))=-1$, then there are exactly two points outside of $\G(f|_{\Fq})$ in its preimage: $\{\lambda_\alpha, -\lambda_\alpha\}\in\beta\Fq$ such that $\lambda_\alpha^2=\frac{\alpha}{(a-1)b}$. In addition, $\pm\lambda_\alpha$ has no preimage.

In particular, we can describe precisely the pre-cycle of the fixed point $\langle\frac{1}{a+1},0\rangle$. First, assuming $q\equiv3\pmod4$, notice that the preimage of $\langle\frac{1}{a+1},0\rangle$ has $4$ elements. As seen in the proof of Theorem $\ref{pre}$, they are $\langle-\frac{1}{a+1},0\rangle$ and $\langle0,\pm\lambda\rangle$, for $\lambda^2=\frac{1}{(a^2-1)b}=$. Since $\chi_2(1-a^2)=1$, by part $2$ we have $\#f^{-1}(0,\pm\lambda)=0$. As $\chi_2(-1)=-1$, then by part $1$, we have $\#f^{-1}(-\frac{1}{a+1},0)=0$. In conclusion, the connected component of $\frac{1}{a+1}$ is isomorphic to $\Zt(4)$.

Now, for $q\equiv1\pmod4$, the preimage of $\frac{1}{a+1}$ is $$f^{-1}\Big(\frac{1}{a+1},0\Big)=\Big\{\Big\langle\frac{1}{a+1},0\Big\rangle,\Big\langle-\frac{1}{a+1},0\Big\rangle\Big\}.$$ 
Since $-1$ is a square, the preimage of $\langle-\frac{1}{a+1},0\rangle$ has $2$ elements: 

 $$f^{-1}\Big(-\frac{1}{a+1},0\Big)=\Big\{\Big\langle\pm\frac{\zeta_4}{a+1},0\Big\rangle\Big\},$$
where $\zeta_n$ is a $n$-th root of unity in $\Fq$. 

If $q\equiv5\pmod8$, then $\zeta_4$ is not a square and $$f^{-1}\Big(\frac{\zeta_4}{a+1},0\Big)=\{\langle0,\pm\lambda\rangle\},$$ for $\lambda^2=\frac{\zeta_4}{(a^2-1)b}.$ Observe that this tree does not go up any more levels, since $ \chi_2(-a^2+1)=1$ and part $3$ of the theorem implies that $\#f^{-1}(0,\pm\lambda)=0$.

If $q\equiv1\pmod8$, then $\zeta_4$ is a square and $$f^{-1}\Big(\frac{\zeta_4}{a+1},0\Big)=\Big\{\Big\langle\pm\frac{\zeta_8}{a+1},0\Big\rangle\Big\}.$$
Following by an induction process similar to Theorem $\ref{precyc}$, we obtain the result below.

\begin{Th}
Let $q-1=2^sr$, with $r$ odd, and $f(X)= c(X^{q+1}+aX^2) $, where $a\neq\{\pm1\}$. If  $\chi_2(1-a^2)=1$, then the connected components that contain the elements of $\Fq$ can be obtained by attaching two nodes to every point $\alpha\in\G(f|_{\Fq})$ that satisfies $\chi_2(\alpha(a-1))=-1$. In particular, $0$ is an isolated fixed point and $\frac 1{a+1}$ is a fixed point with connected component isomorphic to $\Zt(4)$ if $s=1$, or isomorphic to $(Cyc(1),\T(s+1))$ if $s\geq 2$.
\end{Th}

\begin{Ex}
For $q=13$ and $a=2$, notice that $s=2$ and that $\chi_2(1-2^2)=\chi_2(10)=1$. Then the connected components of the elements in $\Fq$ are

\begin{center}
    \includegraphics[scale=0.5]{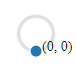}
\end{center}

\begin{center}
    \includegraphics[scale=0.5]{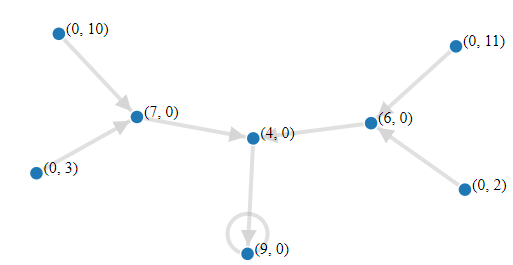}
\end{center}

\begin{center}
    \includegraphics[scale=0.6]{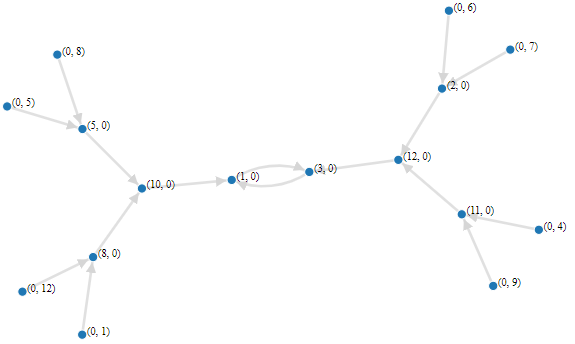}
\end{center}

The component that contains zero is an isolated cycle, the component of $\frac{1}{a+1}=9$ is isomorphic to $(Cyc(1), \T(3))$ and all the points $\alpha\in\Fq$ that satisfy $\chi_2(\alpha(a-1))=-1$ have two points in $\beta\Fq$ directed to them.
\end{Ex}

When $\chi_2(1-a^2)=-1$, the conditions that control the growth of such components are harder to obtain. For example, if $q\equiv3\pmod4$, to make this trees go up exactly one more level, i.e. have the elements in $\beta\Fq$, attached to the points $\alpha\in\G(f|_{\Fq})$ that satisfy $\chi_2(\alpha(a-1))=-1$, have two preimages which have no preimage, we must ask $ba^2+(a+1)^2$ to be a square in $\Fq$.

In the following examples we will give the component that contains $\frac{1}{a+1}$ for $\chi_2(1-a^2)=-1$.

\begin{Ex}
Let $q=19$ and $a=5$. Notice that $\chi_2(1-5^2)=-1$. Then the connected component of the fixed point $6^{-1}=16$ in the functional graph of $f(X)=X^{q+1}+5X^2$ over $\mathbb{F}_{19^2}$ is isomorphic to 
\begin{center}
\includegraphics[scale=0.3]{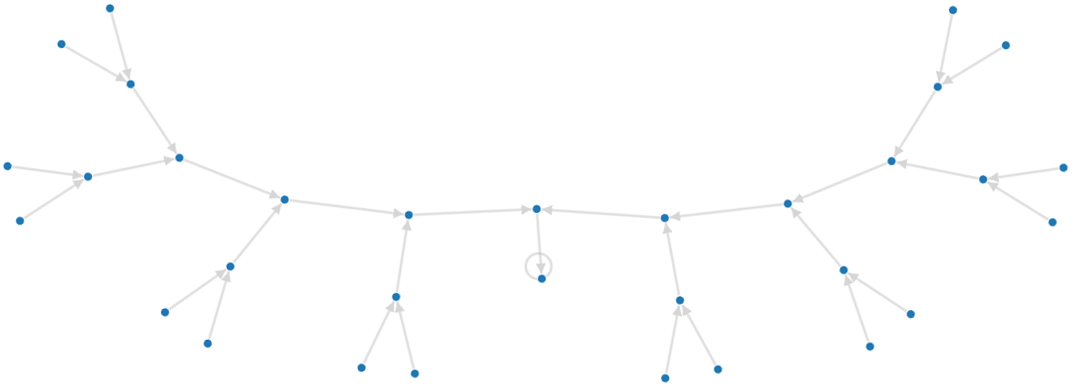}
\end{center}
\end{Ex}

\begin{Ex}
Let $q=13$ and $a=3$. Once more $\chi_2(1-3^2)=-1$. Then the connected component of the fixed point $4^{-1}=10$ in the functional graph of $f(X)=X^{q+1}+3X^2$ over $\mathbb{F}_{13^2}$ is isomorphic to
\begin{center}
    \includegraphics[scale=0.3]{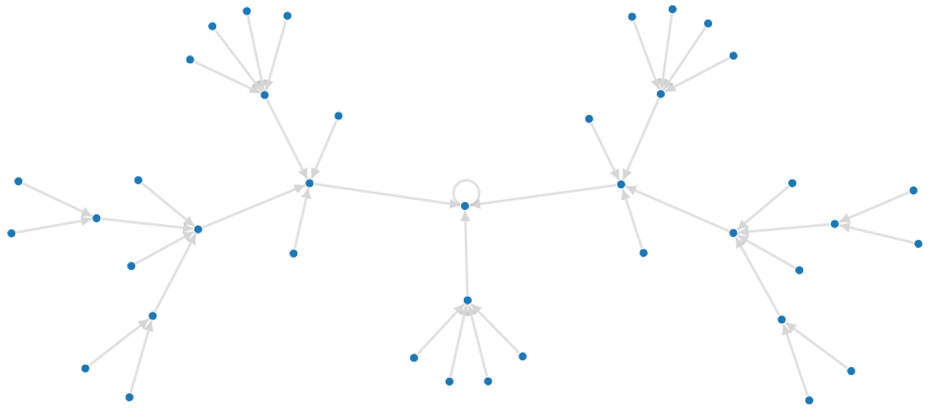}
\end{center}
\end{Ex}

\end{document}